\documentclass[11pt,a4paper,reqno]{amsart}

\usepackage{amsmath}
\usepackage{amsfonts}
\usepackage{a4wide}
\usepackage{amssymb}
\usepackage{amsthm}
\usepackage{mathrsfs}
\usepackage[colorlinks, citecolor=blue, linkcolor=red]{hyperref}

\theoremstyle{plain}

\newtheorem{teo}{Theorem}[section]
\newtheorem{lemma}[teo]{Lemma}
\newtheorem{prop}[teo]{Proposition}

\newtheorem{ackn}{Acknowledgments\!}

\theoremstyle{definition}

\theoremstyle{remark}

\numberwithin{equation}{section}

\def\SS{{{\mathbb S}}}

\def\RR{{\mathbb R}}

\def\RRR{{\mathrm R}}

\def\WWW{{\mathrm W}}

\def\CCC{{\mathrm C}}

\def\M1{\mathscr{M}_{1}}

\def\eps{\varepsilon}

\title[On conformally flat manifolds with constant positive  scalar curvature]{On conformally flat manifolds \\with constant positive scalar curvature}

\author[Giovanni Catino]{Giovanni Catino}
\address[Giovanni Catino]{Dipartimento di Matematica, Politecnico di Milano, Piazza Leonardo da Vinci 32, 20133 Milano, Italy}
\email[]{giovanni.catino@polimi.it}

\begin{document}

\begin{abstract} We classify compact conformally flat $n$-dimensional manifolds with constant positive scalar curvature and satisfying an optimal integral pinching condition: they are covered isometrically by either $\SS^{n}$ with the round metric, $\SS^{1}\times \SS^{n-1}$ with the product metric or $\SS^{1}\times \SS^{n-1}$ with a rotationally symmetric Derdzi\'nski metric.
\end{abstract}

\maketitle

\begin{center}


\medskip

\centerline{\bf AMS subject classification:  53C20, 53C21}

\end{center}

\section{Introduction}

\noindent 

In this paper, we study compact conformally flat Riemannian manifolds, i.e. compact manifolds whose metrics are locally conformally equivalent to the Euclidean metric. Riemannian surfaces are always conformally flat, hence it is natural to look to the higher-dimensional case. Kuiper~\cite{kuiper} was the first who studied global properties of this class of manifolds. He showed that every compact, simply connected, conformally flat manifolds is conformally diffeomorphic to the round sphere $\SS^{n}$. In the last years, much attention has been given to the classification of conformally flat manifolds under topological and/or geometrical assumptions. From the curvature point of view, conformal flatness is equivalent to the vanishing of the Weyl and the Cotton tensor. In particular, the Riemann tensor can be recovered by its trace part, namely the Ricci tensor. Schoen and Yau~\cite{schyau} showed that conformal flatness together with (constant) positive scalar curvature still allows much flexibility. In contrast, conditions on the Ricci curvature put strong restrictions on the geometry of the manifold. Tani~\cite{tani} proved that any compact conformally flat $n$-dimensional manifold with positive Ricci curvature and constant positive scalar curvature is covered isometrically by $\SS^{n}$ with the round metric. This result, with a {\em pointwise pinching} condition on the Ricci curvature, was generalized by many authors (for instance see~\cite{goldberg, ryan, noronha, zhu, cheng} for results and references). In~\cite{carrherz} Carron and Herzlich classify complete conformally flat manifolds of dimension $n\geq 3$ with non-negative Ricci curvature: they are either flat, or locally isometric to $\RR\times\SS^{n-1}$ with the product metric; or are globally conformally equivalent to $\RR^{n}$ or to a spherical space form. On the other hand, classification of compact conformally flat manifolds satisfying an {\em integral pinching} condition were obtained by Gursky~\cite{gursky1} and Hebey and Vaugon~\cite{hebvau2, hebvau}. They showed that $n$-dimensional spherical space form can be characterized by means of optimal $L^{n/2}$-pinching  condition on the Ricci curvature (see~\cite{pigrigset, xuzhao} for other results in this direction).

The aim of this paper is to show a new classification result for compact conformally flat $n$-dimensional manifolds with constant positive scalar curvature. A large variety of Riemannian manifolds belong to this class: manifolds which are covered isometrically by $\SS^{n}$ with the round metric or by $\SS^{1}\times\SS^{n-1}$ with the the product metric; but also quotient of $\SS^{n-k}\times\mathbb{H}^{k}$, $2k<n$ with the product metric. In general, by the work of Schoen~\cite{schoen}, one can construct conformally flat manifolds with constant positive scalar curvature by gluing copies of metrics of this type. Additional examples were constructed by Derdzi\'nski~\cite{derdz4} in the class of warped product metrics. More precisely, he showed that there exists a family of warped product metrics (with $(n-1)$-dimensional Einstein fibers) with harmonic curvature, the Ricci tensor of which is not parallel and has less than three distinct eigenvalues at any point, and proved that every manifold with these properties is covered isometrically by one of these examples (see Theorem~\ref{t-der}). When the fibers are isometric to $\SS^{n-1}$, these metrics give rise to examples of compact conformally flat manifolds with constant positive scalar curvature that we will call {\em rotationally symmetric Derdzi\'nski metrics}.

We prove that these metrics, together with the trivial ones, can be characterized as conformally flat metrics with constant positive scalar curvature and satisfying an optimal integral pinching condition. Let $(M^{n},g)$ be a $n$-dimensional Riemannian manifold. We denote by $Ric$ and $R$ the Ricci and the scalar curvature, respectively and by $E$ the trace-less Ricci tensor, i.e. $E=Ric-\frac{1}{n}R\,g$.

Our main result reads as follows:

\begin{teo} \label{t-main}
Let $(M^{n},g)$ be a compact conformally flat $n$-dimensional manifold with constant positive scalar curvature. Then
$$
\int_{M^{n}} |E|^{\frac{n-2}{n}} \Big( R - \sqrt{n(n-1)} |E| \Big) \,\leq \, 0  
$$
and equality occurs if and only if $(M^{n},g)$ is covered isometrically by either $\SS^{n}$ with the round metric, $\SS^{1}\times \SS^{n-1}$ with the product metric or $\SS^{1}\times \SS^{n-1}$ with a rotationally symmetric Derdzi\'nski metric. 
\end{teo}

Since $E\equiv 0$ on $\SS^{n}$ with the round metric, while $R \equiv \sqrt{n(n-1)}|E|$ on $\SS^{1}\times \SS^{n-1}$ with the product metric, Theorem~\ref{t-main} can be interpreted as a rigidity result for an interpolation curvature estimate.

\

\section{Codazzi tensors with constant trace}

Let $(M^{n},g)$ be a smooth Riemannian manifold of dimension $n\geq 3$ and consider a {\em Codazzi tensor} $T$ on $M^{n}$, i.~e., a symmetric bilinear form satisfying the Codazzi equation 
$$
(\nabla_{X} T)(Y,Z)=(\nabla_{Y} T)(X,Z)\, ,
$$ 
for every tangent vectors $X,Y,Z$. For an overview on manifolds admitting a Codazzi tensor see~\cite[Chapter 16.C]{Besse}. In all this section we will assume that $T$ has constant trace. In particular, the trace-free tensor $T'=T-\frac{1}{n}\hbox{tr}(T)\,g$ is again a Codazzi tensor. In a local coordinate system, we have
\begin{equation}\label{eq-cod}
\nabla_{k} T'_{ij} \,=\, \nabla_{j} T'_{ik} \,.
\end{equation}
Throughout the article, the Einstein convention of summing over the repeated indices will be adopted. Taking the covariant derivative of the Codazzi equation and tracing we obtain
\begin{eqnarray*}
 \Delta T'_{ij} &=& \nabla_{k} \nabla_{j} T'_{ik} \\
 &=& \nabla_{j} \nabla_{k} T'_{ik} - R_{ikjl} T'_{kl} + R_{jk} T'_{ik} \,,
\end{eqnarray*}
where we have used the commutation rules of covariant derivatives of symmetric two tensors. Here $R_{ikjl}$ and $R_{jk}$ denote the components of the Riemann and Ricci tensor respectively. Now, since $T'$ is trace-free, from~\eqref{eq-cod} one has $\nabla_{k} T'_{ik} = \nabla_{i} T'_{kk} = 0$. Thus, any trace-free Codazzi tensor $T'$ satisfies the following elliptic system
\begin{equation}\label{eq-ell}
\Delta T'_{ij} \,=\, - R_{ikjl} T'_{kl} + R_{jk} T'_{ik}\,.
\end{equation}
In particular, the following Weitzenb\"ock formula holds
\begin{equation}\label{eq-wei}
\frac{1}{2}\Delta |T'|^{2} \,=\, |\nabla T|^{2} - R_{ikjl} T'_{ij}T'_{kl} + R_{jk} T'_{ij}T'_{ik} \,.
\end{equation}
Using~\eqref{eq-wei}, Berger and Ebin~\cite{berebi} showed that every Codazzi tensor $T$ with constant trace on a compact Riemannian manifold with non-negative sectional curvature is parallel. Morever, if the sectional curvature are positive at some point, then $T'$ must vanish. 

The aim of this section is to show a vanishing theorem for Codazzi tensor with constant trace, in the spirit of the work Gursky~\cite{gursky2} on conformal vector fields. In this paper the author proved a vanishing theorem for conformal vector fields on four-manifold with negative scalar curvature and satisfying and integral pinching condition on the Ricci tensor (see also~\cite{huli2} for other results in this direction). The proof of this result is an improvement of the Bochner method~\cite{bochner} and strongly relies on the validity of a refined Kato-type inequality involving the covariant derivative of the conformal vector field. As first observed by Bourguignon~\cite{jpb2}, trace-free Codazzi tensor satisfies the following sharp inequality (for a proof, see for instance~\cite{hebvau}).
\begin{lemma}\label{l-kat}
Let $T'$ be a trace-free Codazzi tensor on a Riemannian manifold $(M^{n},g)$ and let $\Omega_{0}=\{p\in M^{n}:\, |T'|(p)\neq 0\}$. Then, on $\Omega_{0}$, 
$$
|\nabla T'|^{2} \, \geq \, \frac{n+2}{n} |\nabla |T'| |^{2} \,.
$$
\end{lemma}
From the previous equation, on $\Omega_{0}$, we therefore have
\begin{equation}\label{eq-inq}
\frac{1}{2}\Delta |T'|^{2} \,\geq\, \frac{n+2}{n}|\nabla |T'||^{2} - R_{ikjl} T'_{ij}T'_{kl} + R_{jk} T'_{ij}T'_{ik} \,.
\end{equation}
Since we want to prove an integral estimate, we will need to apply~\eqref{eq-inq} on the whole $M^{n}$. To do this we use the following:
\begin{lemma}\label{l-kaz}
Let $T'$ be a, non-trivial, trace-free Codazzi tensor on the Riemannian manifold $(M^{n},g)$ and let $\Omega_{0}=\{p\in M^{n}:\, |T'|(p)\neq 0\}$. Then $\hbox{Vol}\,(M^{n}\setminus\Omega_{0}) = 0$. In particular~\eqref{eq-inq} holds in an $H^{1}$-sense on $M^{n}$.
\end{lemma}
\begin{proof} The lemma follows from equation~\eqref{eq-ell} and the unique continuation result of Kazdan~\cite[Theorem 1.8]{kazdan} for solutions of elliptic system on Riemannian manifolds.
\end{proof}

The main result of this section is the following integral inequality on trace-free Codazzi tensor.

\begin{prop}\label{p-est}
Let $T'$ be a, non-trivial, trace-free Codazzi tensor on the Riemannian manifold $(M^{n},g)$. For $\eps>0$, define $\Omega_{\eps}=\{p\in M^{n}:\, |T'|(p)\geq \eps \}$, and 
$$
f_{\eps} \,=\left\{
\begin{array}{ccc}
|T'|(p) &\hbox{if} & p\in \Omega_{\eps} \\
\eps &\hbox{if}& p\in M^{n}\setminus\Omega_{\eps} \,.
\end{array} \right.
$$
Then
$$
\int_{M^{n}} \big(- R_{ikjl} T'_{ij}T'_{kl} + R_{jk} T'_{ij}T'_{ik}\big) \, f_{\eps}^{-\frac{n+2}{n}} \,\leq\, 0 \,.
$$
\end{prop}
\begin{proof}
 Multiplying both side of inequality~\eqref{eq-inq} by $f_{\eps}^{\frac{n+2}{n}}$ and integrating by parts, we get
\begin{eqnarray*}
0 &\geq & -\frac{1}{2}\int_{M^{n}} \Delta |T'|^{2} f_{\eps}^{-\frac{n+2}{n}} + \frac{n+2}{n} \int_{M^{n}} |\nabla |T'||^{2} f_{\eps}^{-\frac{n+2}{n}} \\
&& + \int_{M^{n}} \big(- R_{ikjl} T'_{ij}T'_{kl} + R_{jk} T'_{ij}T'_{ik}\big) \,f_{\eps}^{-\frac{n+2}{n}}\\
&=& \frac{1}{2}\int_{M^{n}} \langle \nabla |T'|^{2}, \nabla \big(f_{\eps}^{-\frac{n+2}{n}}\big)\rangle + \frac{n+2}{n} \int_{M^{n}} |\nabla |T'||^{2} f_{\eps}^{-\frac{n+2}{n}} \\
&& + \int_{M^{n}} \big(- R_{ikjl} T'_{ij}T'_{kl} + R_{jk} T'_{ij}T'_{ik}\big) \,f_{\eps}^{-\frac{n+2}{n}}\\
&=& -\frac{n+2}{n}\int_{M^{n}} \langle \nabla |T'|, \nabla f_{\eps}\rangle |T'| \, f_{\eps}^{-\frac{2}{n}} + \frac{n+2}{n} \int_{M^{n}} |\nabla |T'||^{2} f_{\eps}^{-\frac{n+2}{n}} \\
&& + \int_{M^{n}} \big(- R_{ikjl} T'_{ij}T'_{kl} + R_{jk} T'_{ij}T'_{ik}\big) \,f_{\eps}^{-\frac{n+2}{n}} \,.
\end{eqnarray*}
Since $f_{\eps}=|T'|$ on $\Omega_{\eps}$ and $\nabla f_{\eps} =0$ on $M^{n}\setminus \Omega_{\eps}$, we obtain
\begin{eqnarray*}
0 &\geq &  -\frac{n+2}{n}\int_{M^{n}} |\nabla f_{\eps}|^{2}  \, f_{\eps}^{-\frac{n+2}{n}} + \frac{n+2}{n} \int_{M^{n}} |\nabla |T'||^{2} f_{\eps}^{-\frac{n+2}{n}} \\
&& + \int_{M^{n}} \big(- R_{ikjl} T'_{ij}T'_{kl} + R_{jk} T'_{ij}T'_{ik}\big) \,f_{\eps}^{-\frac{n+2}{n}} \\
&=&  \frac{n+2}{n} \int_{M^{n} \setminus \Omega_{\eps}} |\nabla |T'||^{2} f_{\eps}^{-\frac{n+2}{n}} \\
&& + \int_{M^{n}} \big(- R_{ikjl} T'_{ij}T'_{kl} + R_{jk} T'_{ij}T'_{ik}\big) \,f_{\eps}^{-\frac{n+2}{n}} \\
&\geq & \int_{M^{n}} \big(- R_{ikjl} T'_{ij}T'_{kl} + R_{jk} T'_{ij}T'_{ik}\big) \,f_{\eps}^{-\frac{n+2}{n}} \,.
\end{eqnarray*}
This concludes the proof.
\end{proof}

\

\section{Proof of Theorem~\ref{t-main}}

Throughout this section $(M^{n},g)$, $n\geq 3$, is a compact conformally flat Riemannian manifold with constant positive scalar curvature. To fix the notation, we recall the decomposition of the Riemann tensor into the Weyl, the Ricci and the scalar curvature part
$$
R_{ikjl} \,= \, W_{ikjl} + \frac{1}{n-2}\big(R_{ij}g_{kl}-R_{il}g_{jk}+R_{kl}g_{ij}-R_{jk}g_{il}\big) - \frac{R}{(n-1)(n-2)}\big(g_{ij}g_{kl}-g_{il}g_{jk}\big) \,.
$$
Since $g$ is conformally flat, then, if $n\geq 4$, the Weyl tensor must be identically zero. On the other hand, in dimension $n=3$, the Weyl tensor is zero for algebraic reasons and conformally flatness is equivalent to the vanishing of the Cotton tensor
$$
\CCC_{ijk} \,=\, \nabla_{k} \RRR_{ij} - \nabla_{j} \RRR_{ik} -
\tfrac{1}{2(n-1)} \big( \nabla_{k} \RRR \, g_{ij} - \nabla_{j} \RRR \, g_{ik} \big)\,.
$$ 
Moreover, when $n\geq 4$, one has (see \cite[16.3]{Besse})
\begin{equation*}
\,\nabla_{l}\WWW_{ikjl} \,=\, \frac{n-3}{n-2}\CCC_{ijk}.
\end{equation*}
Hence, if we assume that the manifold is conformally flat, both Weyl and Cotton tensor are identically zero. In particular we have that the Schouten tensor
$$
A_{ij} \,=\, \frac{1}{n-2}\big( R_{ij} - \frac{1}{2(n-1)}R g_{ij} \big) 
$$
is a Codazzi tensor with constant trace, since $\hbox{tr}(A) = \frac{1}{2(n-1)} R = \hbox{const}$. This implies that the trace-less Ricci tensor $E = Ric - \frac{1}{n}R\,g$ is a Codazzi tensor too. Assume that $E$ is not identically zero. From Proposition~\ref{p-est}, we get that the following integral inequality holds
\begin{equation}\label{eq-est2}
\int_{M^{n}} \big(- R_{ikjl} E_{ij}E_{kl} + R_{jk} E_{ij}E_{ik}\big) \, f_{\eps}^{-\frac{n+2}{n}} \,\leq\, 0 \,,
\end{equation}
where 
$$
f_{\eps} \,=\left\{
\begin{array}{ccc}
|E|(p) &\hbox{if} & p\in \Omega_{\eps} \\
\eps &\hbox{if}& p\in M^{n}\setminus\Omega_{\eps} \,.
\end{array} \right.
$$
and $\Omega_{\eps}=\{p\in M^{n}:\, |E|(p)\geq \eps \}$. Since $g$ is conformally flat, the Riemann tensor becomes 
$$
R_{ikjl} \,= \, \frac{1}{n-2}\big(E_{ij}g_{kl}-E_{il}g_{jk}+E_{kl}g_{ij}-E_{jk}g_{il}\big) + \frac{R}{n(n-1)}\big(g_{ij}g_{kl}-g_{il}g_{jk}\big) \,,
$$
and a simple computation shows
$$
- R_{ikjl} E_{ij}E_{kl} + R_{jk} E_{ij}E_{ik} \,=\, \frac{1}{n-1} R |E|^{2} + \frac{n}{n-2} E_{ij}E_{ik}E_{jk} \,.
$$
Moreover, since $E$ is trace-free, we have the sharp inequality (see for instance~\cite[Lemma 2.4]{huisk9})
\begin{equation}\label{eq-oku}
E_{ij}E_{ik}E_{jk} \,\geq\, - \frac{n-2}{\sqrt{n(n-1)}}|E|^{3}  
\end{equation}
and equality holds at some point $p\in M^{n}$ if and only if $E$  can be diagonalized at $p$ with $(n-1)$-eigenvalues equal to $\lambda$ and one eigenvalue equals to $-(n-1)\lambda$, for some $\lambda\in\RR$. Hence, we get
$$
- R_{ikjl} E_{ij}E_{kl} + R_{jk} E_{ij}E_{ik} \,\geq\, \frac{1}{n-1} |E|^{2}\Big( R  - \sqrt{n(n-1)} |E| \Big) \,,
$$
and from~\eqref{eq-est2}, we obtain
$$
\int_{M^{n}} |E|^{\frac{n-2}{n}} \Big( R  - \sqrt{n(n-1)} |E| \Big) |E|^{\frac{n+2}{n}}\, f_{\eps}^{-\frac{n+2}{n}} \,\leq \, \int_{M^{n}} \big(- R_{ikjl} E_{ij}E_{kl} + R_{jk} E_{ij}E_{ik}\big) \, f_{\eps}^{-\frac{n+2}{n}} \,\leq\, 0 \,.
$$
Now, taking the limit as $\eps\rightarrow 0$, since $ |E|^{\frac{n+2}{n}}\, f_{\eps}^{-\frac{n+2}{n}} \rightarrow 1$ a.e. on $M^{n}$ by Lemma~\ref{l-kaz}, we conclude
$$
\int_{M^{n}} |E|^{\frac{n-2}{n}} \Big( R  - \sqrt{n(n-1)} |E| \Big) \,\leq \, 0 \,.
$$
Hence, we have proved the following:
\begin{lemma}\label{l-pin}
Let $(M^{n},g)$ be a compact conformally flat manifold with constant positive scalar curvature. Then
$$
\int_{M^{n}} |E|^{\frac{n-2}{n}} \Big( R - \sqrt{n(n-1)} |E| \Big) \,\leq \, 0  
$$
and equality occurs if and only if, at every point, either $E$ is null or it has an eigenvalue of multiplicity $(n-1)$ and another of multiplicity $1$. 
\end{lemma}

Now we can conclude the proof of Theorem~\ref{t-main}. By the integral pinching assumption we have that equality in Lemma~\ref{l-pin} occurs. Hence, either $g$ is Einstein, and by conformally flatness it has constant positive sectional curvature or, at every point, the Ricci tensor has an eigenvalue of multiplicity $(n-1)$ and another of multiplicity $1$. Moreover, since the Ricci tensor is Codazzi, we have that $g$ has harmonic curvature, i.e. $\nabla_{l}R_{ikjl} \equiv 0$ (see~\cite[Chapter 16.E]{Besse}), and by the regularity result of DeTurck and Goldschmidt~\cite{detgol}, $g$ must be real analytic in suitable (harmonic) local coordinates. 

Now, suppose that the Ricci tensor has an eigenvalue of multiplicity $(n-1)$ and another of multiplicity $1$. If the Ricci tensor is parallel, by the de Rham decomposition Theorem~\cite{derham}, $(M^{n},g)$ is covered isometrically by the product of Einstein manifolds. Since $g$ is conformally flat and has positive scalar curvature, then the only possibility is that $(M^{n},g)$ is covered isometrically by $\SS^{1}\times\SS^{n-1}$ with the product metric.

On the other hand, if the Ricci tensor is not parallel, we have the following classification result of Derdzi\'nski~\cite[Theorem 2]{derdz4}:

\begin{teo}\label{t-der} 
Let $(M^{n},g)$ be a compact Riemannian manifold with harmonic curvature. If the Ricci tensor is not parallel and has less than three distinct eigenvalues at each point, then $(M^{n},g)$ is covered isometrically by $(S^{1}\times N^{n-1}, dt^{2}+F^{2}(t)g_{N})$, where $(N^{n-1},g_{N})$ is a compact Einstein manifold with positive scalar curvature and $F$ is a non-constant, positive, periodic function satisfying a precise ODE. Moreover, if $g$ is conformally flat, then $(N^{n-1},g_{N})$ is isometric to $S^{n-1}$ with the round metric.
\end{teo}

This concludes the proof of Theorem~\ref{t-main}. Finally, we recall that splitting results for Riemannian manifolds admitting a Codazzi tensor with only two distinct eigenvalues were obtained by Derdzinski~\cite{derdz3}, Hiepko-Reckziegel~\cite{hiepkoreck} (see~\cite[Chapter~16]{Besse} for further discussion). See also a more recent result of the author with Mantegazza and Mazzieri~\cite{catmanmaz}.

\

\

\

\begin{ackn} The author is members of the Gruppo Nazionale per
l'Analisi Matematica, la Probabilit\`{a} e le loro Applicazioni (GNAMPA) of the Istituto Nazionale di Alta Matematica (INdAM) and is supported by the GNAMPA project ``Equazioni di evoluzione geometriche e strutture di tipo Einstein''. The author would like to thank Matthew Gursky for several valuable suggestions.
\end{ackn}

\

\bibliographystyle{amsplain}
\bibliography{biblio}

\

\parindent=0pt

\

\end{document}